\documentclass[a4paper,11pt,twoside]{amsart}
\usepackage{amsmath}
\usepackage[all]{xy}

\theoremstyle{plain}
\newtheorem{lemma}{Lemma}[section]
\newtheorem{prop}[lemma]{Proposition}
\newtheorem{theo}[lemma]{Theorem}
\newtheorem{coro}[lemma]{Corollary}

\theoremstyle{remark}
\newtheorem{rem}[lemma]{Remark}

\theoremstyle{definition}
\newtheorem{definition}[lemma]{Definition}
\newtheorem{ex}[lemma]{Example}

\def\hom{\mathrm{Hom}}

\def\cn{\mathrm{C}}

\def\im{\mathrm{Im}\:}

\def\N{\mathbb{N}}
\def\Z{\mathbb{Z}}

\def\ext{\mathrm{Ext}}

\begin{document}

\title{$A$-homology, $A$-homotopy and spectral sequences}

\author[E.M. Ottina]{Enzo Miguel Ottina}

\address{Enzo Miguel Ottina. Instituto de Ciencias B\'asicas \\
 Universidad Nacional de Cuyo \\ Mendoza, Argentina.}

\email{emottina@uncu.edu.ar}

\begin{abstract}
\noindent
Given a CW-complex $A$ we define an `$A$-shaped' homology theory which behaves nicely towards $A$-homotopy groups allowing the generalization of many classical results. We also develop a relative version of the Federer spectral sequence for computing $A$-homotopy groups. As an application we derive a generalization of the Hopf-Whitney theorem.
\end{abstract}



\maketitle

\section{Introduction} \label{intro}

Given pointed topological spaces $A$ and $X$, the $A$-homotopy groups of $X$ are defined as $\pi_n^A(X)=[\Sigma^n A,X]$, that is, the homotopy classes of pointed maps from the reduced $n$-th suspension of $A$ to $X$. These groups appear naturally in different situations, for example in Quillen's model categories \cite{Qui}, and generalize the homotopy groups with coefficients \cite{Nei,Pet}.

The $A$-homotopy groups have also been studied indirectly by some authors as homotopy groups of function spaces. Among them, we mention M. Barratt, R. Brown and H. Federer. The first one proves in \cite{Bar1} that homotopy groups of function spaces can be described as a central extension of certain homology groups, and computes in \cite{Bar2} these homotopy groups in several cases using Whitney's tube systems \cite{Whn}. In \cite{Br}, Brown works in a simplicial setting obtaining results which are used to study homotopy types of function spaces. As an application, he obtains different proofs for some of Barratt's results. In \cite{Fed}, Federer introduces  a spectral sequence which converges to the homotopy groups of function spaces. Clearly, the Federer spectral sequence may also be understood as a tool to compute $A$-homotopy groups when $A$ is locally compact and Hausdorff.

In the first part of this article we delve deeply into this spectral sequence taking a different approach to that of Federer's: we focus our attention on $A$-homotopy groups of spaces rather than on homotopy groups of function spaces. We develop a relative version of the Federer spectral sequence and obtain as a first application a generalization of the Hopf-Whitney theorem (\ref{Hopf-Whitney}).

The homotopy groups and the homology groups of a topological space are related, for example, by the Hurewicz theorem, or more generally, by the Whitehead exact sequence \cite{Whi}. Therefore, it is natural to think that the $A$-homotopy groups should also have their homological counterpart. The main objective of the second part of this article is to define a suitable `$A$-shaped' homology theory and give results which show the relationship between this homology theory and the $A$-homotopy groups. This is achieved in section \ref{section_A_homology} where, given a CW-complex $A$, we define the $A$-homology groups of a CW-complex $X$ generalizing singular homology groups. We obtain many generalization of classical results, among the most important of which we mention a Hurewicz-type theorem relating the $A$-homotopy groups with the $A$-homology groups (\ref{Hurewicz_for_A-homology}) and a homological version of the Whitehead theorem which states that, under certain hypotheses, a map between CW-complexes which induces isomorphisms in the $A$-homology groups is a homotopy equivalence (\ref{A-homology_isomorphism}).

Finally, we define a Hurewicz-type map between the $A$-homotopy groups and the $A$-homology groups and embed it in a long exact sequence generalizing the exact sequence constructed by Whitehead in \cite{Whi}.

\medskip

Throughout this article, all spaces are supposed to be pointed and path-connected. Homology and cohomology will mean reduced homology and reduced cohomology respectively. Also, if $X$ is a pointed topological space with base point $x_0$, we will simply write $\pi_n(X)$  instead of $\pi_n(X,x_0)$.

\medskip

I would like to thank G. Minian for many valuable comments and suggestions on this article

\section{A relative version of the Federer spectral sequence} \label{section_Federer}

In this section we introduce a relative version of the Federer spectral sequence which will be used later. As one of its first applications, we will obtain a relative version of the Hopf-Whitney theorem.

Recall that the $A$-homotopy groups of a (pointed) topological space $X$ are defined by $\pi_n^A(X)=[\Sigma^n A,X]$, that is the (pointed) homotopy classes of maps from $\Sigma^n A$ to $X$. Similarly, the relative $A$-homotopy groups of a (pointed) topological pair $(Y,B)$ are defined by $\pi_n^A(Y,B)=[(\cn \Sigma^{n-1} A,\Sigma^{n-1} A),(Y,B)]$.

Now we state and prove the main result of this section.

\begin{theo} \label{theo_relative_Federer_Spectral_Sequence}
Let $(Y,B)$ be a topological pair such that $\pi_2(Y,B)$ is an abelian group and let $A$ be a finite-dimensional (and path-connected) CW-complex. Then there exists a homological spectral sequence $\{E^a_{p,q}\}_{a\geq 1}$, with $E^2_{p,q}$ satisfying
\begin{itemize}
\item $E^2_{p,q}\cong H^{-p}(A;\pi_q(Y,B))$ for $p+q\geq 2$ and $p\leq -1$.
\item $E^2_{p,q}$ is isomorphic to a subgroup of $H^{-p}(A;\pi_{q}(Y,B))$ if $p+q=1$ and $p\leq -1$.
\item $E^2_{p,q}=0$ if $p+q\leq 0$ or $p\geq 0$.
\end{itemize}
which converges to $\pi_{p+q}^A(Y,B)$ for $p+q\geq 2$.
\end{theo}

We will call $\{E^a_{p,q}\}_{a\geq 1}$ the \emph{relative Federer spectral sequence associated to $A$ and $(Y,B)$}.

\begin{proof}
We may suppose that $A$ has only one $0$-cell. For $r\leq -1$, let $A^r=\ast$, and for $r\in\N$, let $J_r$ be an index set for the $r$-cells of $A$. For $\alpha \in J_r$ let $g^r_\alpha$ be the attaching map of the cell $e^r_\alpha$.

For $r\in\mathbb{N}$ let $\displaystyle Z_r= \bigvee_{J_r}S^r \cong A^r/A^{r-1}$. The long exact sequences associated to the cofiber sequences $\displaystyle A^{r-1}\to A^r \overset{\overline{q}_r}{\to} Z_r$, $r\in\mathbb{N}$, may be extended as follows
\begin{displaymath}
\xymatrix@C=12pt{\ldots \ar[r] & \pi^{A^{r-1}}_2(Y,B) \ar[r]^(.53){\partial_r} & \pi^{Z_r}_1(Y,B) \ar[r]^(.48){\eta}  & \dfrac{\pi^{Z_r}_1(Y,B)}{\im\partial_r} \ar[r]^(.45){0} & \dfrac{\pi^{Z_{r-1}}_1(Y,B)}{\im\partial_{r-1}} \ar[r]^{\textrm{id}} & \dfrac{\pi^{Z_{r-1}}_1(Y,B)}{\im\partial_{r-1}} \ar[r] & 0}
\end{displaymath}
where $\eta$ is the quotient map.

These extended exact sequences yield an exact couple $(A_0,E_0,i,j,k)$ where the bigraded groups $\displaystyle A_0 = \bigoplus_{p,q \in \Z} A^1_{p,q}$ and $\displaystyle E_0 = \bigoplus_{p,q \in \Z} E^1_{p,q}$ are defined by
\begin{displaymath}
A^1_{p,q}=\left\{\begin{array}{ll}\pi_{p+q+1}^{A^{-p-1}}(Y,B) & \textrm{if $p+q\geq 1$} \\ \pi_1^{Z_{-p-1}}(Y,B)/\im \partial_{-p-1} & \textrm{if $p+q=0$} \\  0 & \textrm{if $p+q\leq -1$} \end{array}\right.
\end{displaymath}
and
\begin{displaymath}
E^1_{p,q}=\left\{\begin{array}{ll}\pi_{p+q}^{Z_{-p}}(Y,B) & \textrm{if $p+q\geq 1$} \\ \pi_1^{Z_{-p-1}}(Y,B)/\im \partial_{-p-1} & \textrm{if $p+q=0$} \\  0 & \textrm{if $p+q\leq -1$} \end{array}\right.
\end{displaymath}

Note that all these groups are abelian, except perhaps for $\pi_2^{A^r}(Y,B)$, $r\in\N$. Hence, $E_0$ is an abelian group. Therefore, the exact couple $(A_0,E_0,i,j,k)$ induces a spectral sequence $(E^a_{p,q})_{p,q}$, $a\geq 1$, which converges to $\pi_n^A(Y,B)$ for $n\geq 2$ since $A$ is finite-dimensional.

Note also that 
$$\displaystyle E^1_{p,q} = \pi_{p+q}^{Z_{-p}}(Y,B) = \prod_{J_{-p}} \pi_q(Y,B) \cong C^{-p}(A;\pi_q(Y,B))$$
for $p+q\geq 1$ and $p\leq -1$, where $C^{\ast}(A;\pi_q(Y,B))$ denotes the cellular cohomology complex of $A$ with coefficients in $\pi_q(Y,B)$.

The isomorphism $\gamma:E^1_{p,q}=\pi_{p+q}^{Z_{-p}}(Y,B)\to C^{-p}(A;\pi_q(Y,B))$ is given by $$\gamma([f])(e_\alpha^{-p})=[f\circ\cn\Sigma^{p+q-1}i_\alpha]$$
where $i_\alpha:S^{-p}\to Z_{-p}$ denotes the inclusion in the $\alpha$-th copy of $S^{-p}$. Note also that $E^2_{p,q}=0$ if $p+q\leq 0$ or $p\geq 0$.

We wish to prove now that $E^2_{p,q}\cong H^{-p}(A;\pi_q(Y,B))$ for $p+q\geq 2$ and $p\leq -1$. 

We consider the morphism $\delta:E^1_{p,q}\cong C^{-p}(A;\pi_q(Y,B))\to E^1_{p-1,q}\cong C^{-p+1}(A;\pi_q(Y,B))$ coming from the spectral sequence. We will prove that $\delta=d^\ast$ for $n= p+q\geq 2$ and $p\leq -1$, where $d^\ast$ is the cellular boundary map. This is equivalent to saying that the following diagram commutes
\begin{displaymath}
\xymatrix{\pi^{Z_{p'}}_{n}(Y,B) \ar[r]^{(\overline{q}_{p'})^\ast} \ar[d]^\cong_\gamma & \pi^{A^{p'}}_{n}(Y,B) \ar[r]^(.45){(\underset{J_{p'+1}}{+} g_\beta^{p'+1})^\ast } &  \pi^{Z_{p'+1}}_{n-1}(Y,B) \ar[d]^\cong_\gamma \\ C^{p'}(A;\pi_{n+p'}(Y,B)) \ar[rr]^{d^\ast} & & C^{p'+1}(A;\pi_{n+p'}(Y,B)) }
\end{displaymath}
where $p'=-p$.

If $[h]\in \pi^{Z_{p'}}_{n}(Y,B)$ and $e_\alpha^{p'+1}$ is a $(p'+1)$-cell of $A$, then
\begin{displaymath}
\begin{array}{rcl} \left(\gamma(\underset{\beta\in J_{p'+1}}{+} g_\beta^{p'+1})^\ast q^\ast(h)\right)(e_\alpha^{p'+1}) & = & \gamma(h\cn\Sigma^{n-1} q(\underset{\beta\in J_{p'+1}}{+} \cn\Sigma^{n-1} g_\beta^{p'+1}))(e_\alpha^{p'+1}) = \\ & = & [h\cn\Sigma^{n-1} q\cn\Sigma^{n-1} g_\alpha^{p'+1}]. \end{array}
\end{displaymath}

On the other hand,
\begin{displaymath}
\begin{array}{rcl}
\displaystyle d^{\ast}(\gamma([h]))(e_\alpha^{p'+1}) & = & \displaystyle (\gamma([h]))(d(e_\alpha^{p'+1})) = \sum_{\beta\in J_{p'}}\deg(q_\beta g^{p'+1}_\alpha)(\gamma([h]))(e_\beta^{p'})= \\  & = & \displaystyle \sum_{\beta\in J_{p'}}\deg(q_\beta g^{p'+1}_\alpha)[h\cn \Sigma^{n-1} i_\beta]
\end{array}
\end{displaymath}
where $q_\beta:A^{p'}\to S^{p'}$ is the map that collapses $A^{p'}-e^{p'}_\beta$ to a point.

Let $r=n-1+p'$. Since the morphism $\displaystyle \bigoplus_{\beta\in J_{p'}} (\Sigma^{n-1} q_\beta)_\ast$ is the inverse of the isomorphism $\displaystyle \bigoplus_{\beta\in J_{p'}} (\Sigma^{n-1} i_\beta)_\ast:\bigoplus_{\beta\in J_{p'}}\pi_r(S^{r})\to\pi_r(\Sigma^{n-1}Z_{p'})$ we obtain that

\begin{displaymath}
\begin{array}{lcl}
[\Sigma^{n-1}\overline{q}_{p'}\Sigma^{n-1}g_{\alpha}^{p'+1}] & = & \displaystyle \bigoplus_{\beta\in J_{p'}} (\Sigma^{n-1} i_\beta)_\ast ( \bigoplus_{\beta\in J_{p'}} (\Sigma^{n-1} q_\beta)_\ast([\Sigma^{n-1}\overline{q}_{p'}\Sigma^{n-1}g_{\alpha}^{p'+1}])) = \\ & = & \displaystyle \bigoplus_{\beta\in J_{p'}} (\Sigma^{n-1} i_\beta)_\ast (\{[\Sigma^{n-1} q_\beta \Sigma^{n-1} g_{\alpha}^{p'+1}]\}_{\beta\in J_{p'}}) = \\ & = & \displaystyle \sum_{\beta\in J_{p'}}[\Sigma^{n-1} i_\beta \Sigma^{n-1} q_\beta \Sigma^{n-1} g_{\alpha}^{p'+1}].
\end{array}
\end{displaymath}

Hence,

\begin{displaymath}
\begin{array}{rcl}
[h\cn \Sigma^{n-1} \overline{q}_{p'}\cn \Sigma^{n-1} g_\alpha^{p'+1}] & = & \displaystyle h_\ast([\cn \Sigma^{n-1} \overline{q}_{p'}g_\alpha^{p'+1}]) =  h_\ast(\cn \Sigma^{n-1}(\sum_{\beta\in J_{p'}}[i_\beta q_\beta g_\alpha^{p'+1}])) =  \\ & = & \displaystyle \sum_{\beta\in J_{p'}}[h \cn \Sigma^{n-1} (i_\beta  q_\beta g_\alpha^{p'+1})] = \\ & = & \displaystyle \sum_{\beta\in J_{p'}}[h \cn \Sigma^{n-1} i_\beta][\cn \Sigma^{n-1} (q_\beta g^{p'+1}_\alpha)] = \\ & = & \displaystyle  \sum_{\beta\in J_{p'}}\deg(q_\beta g^{p'+1}_\alpha)[h\cn \Sigma^{n-1}i_\beta].
\end{array}
\end{displaymath}

It follows that $E^2_{p,q}\cong H^{-p}(A;\pi_q(Y))$ for $p+q\geq 2$ and $p\leq -1$.

The same argument works for the case $p+q=1$, $p\leq -2$, and we obtain a commutative diagram
\begin{displaymath}
\xymatrix@C=40pt{\pi^{Z_{p'}}_{1}(Y,B) \ar[r]^{(\overline{q}_{p'})^\ast} \ar[d]^\cong_\gamma & \pi^{A^{p'}}_{1}(Y,B) \ar[r]
^(.45){(\underset{J_{p'+1}}{+} g_\beta^{p'+1})^\ast } &  \pi^{\underset{J_{p'+1}}{\bigvee}S^{p'}}_{1}\!\!\!(Y,B)_{\phantom{\underset{J_{p'}}{\bigvee}S^{p'}}} \ar[d]^\cong \\ C^{p'}(A;\pi_{p'+1}(Y,B)) \ar[rr]^{d^\ast} & & C^{p'+1}(A;\pi_{p'+1}(Y,B)) \ar@{-}[u]+<0 pt,-9 pt>}
\end{displaymath}

Then $E^2_{p,q}=\ker d^1_{p,q}/\im d^1_{p+1,q}=\im \partial_{-p}/\im d^1_{p+1,q}$. By exactness, $\im \partial_{-p}=\ker q^\ast$. Thus, $\im \partial_{-p}\subseteq \ker d^\ast$ since the previous diagram commutes. Moreover, if $p\leq -2$ by the previous case the map $d^1_{p+1,q}$ coincides, up to isomorphisms, with the map $d^\ast:C^{p'-1}(A;\pi_{p'+1}(Y,B)) \to C^{p'}(A;\pi_{p'+1}(Y,B))$, and in case $p=-1$, both maps are trivial. Therefore, $E^2_{p,q}$ is isomorphic to a subgroup of $H^{-p}(A;\pi_{q}(Y,B))$ if $p+q=1$ and $p\leq -1$.
\end{proof}

Of course, applying this theorem to the topological pair $(CY,Y)$ we obtain the following absolute version, which is similar to Federer's result.

\begin{coro} \label{theo_Federer_Spectral_Sequence}
Let $Y$ be a topological space with abelian fundamental group and let $A$ be a finite-dimensional (and path-connected) CW-complex. Then there exists a homological spectral sequence $\{E^a_{p,q}\}_{a\geq 1}$, with $E^2_{p,q}$ satisfying
\begin{itemize}
\item $E^2_{p,q}\cong H^{-p}(A;\pi_q(Y))$ for $p+q\geq 1$ and $p\leq -1$.
\item $E^2_{p,q}$ is isomorphic to a subgroup of $H^{-p}(A;\pi_{q}(Y))$ if $p+q=0$ and $p\leq -1$.
\item $E^2_{p,q}=0$ if $p+q<0$ or $p\geq 0$.
\end{itemize}
which converges to $\pi_{p+q}^A(Y)$ for $p+q\geq 1$.
\end{coro}

We will call $\{E^a_{p,q}\}_{a\geq 1}$ the \emph{Federer spectral sequence associated to $A$ and $Y$}.

\bigskip

Note that the relative version of the Federer spectral sequence is natural in the following sense. If $A$ is a finite-dimensional CW-complex, $(Y,B)$ and $(Y',B')$ are topological pairs such that the groups $\pi_2(Y,B)$ and $\pi_2(Y',B')$ are abelian and $f:(Y,B)\to (Y',B')$ is a continuous map, then $f$ induces a morphism between the relative Federer spectral sequence associated to $A$ and $(Y,B)$ and the one associated to $A$ and $(Y',B')$. Indeed, $f$ induces morphisms between the extended long exact sequences of the proof above and hence a morphism between the exact couples involved, which gives rise to the morphism between the spectral sequences mentioned above.

Moreover, if $A'$ is another finite-dimensional CW-complex and $g:A\to A'$ is a cellular map, then $g$ also induces morphisms between the extended long exact sequences mentioned before and therefore, a morphism between the relative Federer spectral sequence associated to $A$ and $(Y,B)$ and the one associated to $A'$ and $(Y,B)$. If $g$ is not cellular we may replace it by a homotopic cellular map to obtain the induced morphism. Of course, by the description of the second page of our spectral sequence, the map $g$ itself will also induce the same morphism from page two onwards.

Clearly, the same holds for the absolute version.

\begin{rem} \ 

\noindent
(1) Looking at the extended exact sequences of the proof of \ref{theo_relative_Federer_Spectral_Sequence} we obtain that the relative Federer spectral sequence converges to the trivial group in degree $1$. Thus, the groups $E^2_{p,q}$, with $p+q=1$, become all trivial in $E^\infty$.

\noindent
(2) As we have mentioned above, the spectral sequence given by Federer in \cite{Fed} is similar to our absolute version. But since we work with pointed topological spaces our version enables us to compute homotopy groups of function spaces only when the base point is the constant map. However, the hypothesis we require on the space $Y$ ($\pi_1(Y)$ is abelian) are weaker than Federer's ($\pi_1(Y)$ acts trivially on $\pi_n(Y)$ for all $n\in\mathbb{N}$). Moreover, our approach in terms of $A$-homotopy groups admits the relative version given before.
\end{rem}

Just as a simple example of application of the Federer spectral sequence consider the following, which is a reformulation of a well-known result for homotopy groups with coefficients.

\begin{ex} \label{ex_A_homotopy}
If $A$ is a Moore space of type $(G,m)$ (with $G$ finitely generated) and $X$ is a path-connected topological space with abelian fundamental group, in the Federer spectral sequence we get
$$E^2_{-p,q}=\left\{\begin{array}{cl}\hom(G,\pi_q(X)) & \textrm{if $p=m$} \\  \ext(G,\pi_{q}(X)) & \textrm{if $p=m+1$} \\ 0 & \textrm{otherwise} \end{array} \right. \qquad \qquad \textrm{for $-p+q\geq 1$.}$$

Hence, from the corresponding filtrations, we deduce that, for $n\geq 1$, there are short exact sequences of groups

\begin{displaymath}
\xymatrix{0 \ar[r] & \ext(G,\pi_{n+m+1}(X)) \ar[r] & \pi_n^A(X) \ar[r] & \hom(G,\pi_{n+m}(X)) \ar[r] & 0}
\end{displaymath}

As a corollary, if $G$ is a finite group of exponent $r$ then $\alpha^{2r}=0$ for every $\alpha\in\pi_n^A(X)$. For example, if $X$ is a path-connected topological space with abelian fundamental group, then every element in $\pi_n^{\mathbb{P}^2}(X)$ ($n\geq 1$) has order 1, 2 or 4.
\end{ex}

We will now apply \ref{theo_Federer_Spectral_Sequence} to obtain an extension to the Hopf-Whitney theorem.

\begin{theo} \label{Hopf-Whitney}
Let $K$ be a path-connected CW-complex of dimension $n\geq 2$ and let $Y$ be $(n-1)$-connected. Then there exists a bijection $[K,Y]\leftrightarrow H^n(K;\pi_n(Y))$.

In addition, if $K$ is the suspension of a path-connected CW-complex (or if $Y$ is a loop space), then the groups $[K,Y]$ and $H^n(K;\pi_n(Y))$ are isomorphic.

Moreover, this isomorphism is natural in $K$ and in $Y$.
\end{theo}

\begin{proof}
The first part is the Hopf-Whitney theorem (cf. \cite{MT}). The second part can be proved easily by means of the Federer spectral sequence. Concretely, suppose that $K=\Sigma K'$ with $K'$ path-connected. Let $\{E^a_{p,q}\}$ denote the Federer spectral sequence associated to $K'$ and $Y$. Then $E^2_{p,q}=0$ for $q\leq n-1$ since $Y$ is $(n-1)$-connected, and $E^2_{p,q}=0$ for $p\leq -n$ since $\dim K'=n-1$. Hence, $E^2_{-(n-1),n}\cong H^{n-1}(K';\pi_n(Y))$ survives to $E^\infty$. As it is the only nonzero entry in the diagonal $p+q=1$ of $E^2$ it follows that
$$[K,Y]\cong \pi_1^{K'}(Y) \cong E^2_{-(n-1),n} \cong H^{n-1}(K';\pi_n(Y)) \cong H^n(K;\pi_n(Y)).$$

Finally, naturality follows from naturality of the Federer spectral sequence.
\end{proof}

In a similar way, from theorem \ref{theo_relative_Federer_Spectral_Sequence} we obtain the following relative version of the Hopf-Whitney theorem, which not only is interesting for its own sake but also will be important for our purposes.

\begin{theo} \label{relative_Hopf-Whitney}
Let $K$ be the suspension of a path-connected CW-complex of dimension $n-1\geq 1$ and let $(Y,B)$ be an $n$-connected topological pair. Then there exists an isomorphism of groups $$[(\cn K,K);(Y,B)]\leftrightarrow H^n(K;\pi_{n+1}(Y,B)).$$
which is natural in $K$ and in $(Y,B)$.
\end{theo}

\begin{proof}
Suppose that $K=\Sigma K'$ with $K'$ path-connected. Let $\{E^a_{p,q}\}$ denote the relative Federer spectral sequence associated to $K'$ and $(Y,B)$. Then $E^2_{p,q}=0$ for $q\leq n$ since $(Y,B)$ is $n$-connected, and $E^2_{p,q}=0$ for $p\leq -n$ since $\dim K'=n-1$. Hence, $E^2_{-(n-1),n+1}=H^{n-1}(K';\pi_{n+1}(Y,B))$ survives to $E^\infty$. As it is the only nonzero entry in the diagonal $p+q=2$ of $E^2$ it follows that
\begin{displaymath}
\begin{array}{rcl}
[(CK,K);(Y,B)] & = & \pi_2^{K'}(Y,B) \cong E^2_{-(n-1),n+1} \cong H^{n-1}(K';\pi_{n+1}(Y,B)) \cong \\ & \cong & H^n(K;\pi_{n+1}(Y,B)).
\end{array}
\end{displaymath}
and naturality follows again from naturality of the Federer spectral sequence.
\end{proof}

We will give now another application of \ref{theo_Federer_Spectral_Sequence}. We will denote by $\mathcal{T}_{\mathcal{P}}$ the class of torsion abelian groups whose elements have orders which are divisible only by primes in a set $\mathcal{P}$ of prime numbers.

\begin{prop} \label{A-homotopy}
Let $A$ be a finite-dimensional CW-complex such that $H_n(A)$ is finitely generated for all $n\in\N$ and let $X$ be a path-connected topological space such that $\pi_1(X)$ is abelian. If $H_n(A)\in \mathcal{T}_{\mathcal{P}}$ for all $n\in\mathbb{N}$ then $\pi_n^A(X) \in \mathcal{T}_{\mathcal{P}}$ for all $n\in\mathbb{N}$.
\end{prop}

\begin{proof}
By \ref{theo_Federer_Spectral_Sequence}, it suffices to prove that $H^{-p}(A;\pi_q(X))\in \mathcal{T}_{\mathcal{P}}$ for all $p,q\in \mathbb{Z}$ such that $p+q\geq 0$ and $p\leq -1$.

By the universal coefficient theorem $$H^{-p}(A;\pi_q(X))\cong \hom (H_{-p}(A),\pi_q(X)) \oplus \ext (H_{-p-1}(A),\pi_q(X)).$$ Since $A$ is $\mathcal{T}_{\mathcal{P}}$-acyclic and $H_n(A)$ is finitely generated for all $n\in\N$ it follows that $\hom (H_{-p}(A),\pi_q(X))\in \mathcal{T}_{\mathcal{P}}$ and $\ext (H_{-p-1}(A),\pi_q(X))\in \mathcal{T}_{\mathcal{P}}$ for all $p\leq -1$ and $q\geq 0$. Thus, $H^{-p}(A;\pi_q(X)) \in \mathcal{T}_{\mathcal{P}}$ for all $p\leq -1$ and $q\geq 0$.
\end{proof}

\section{$A$-homology} \label{section_A_homology}

In this section we define an `\emph{$A$-shaped}' reduced homology theory, which we call $A$-homology and which coincides with the singular homology theory in case $A=S^0$. Our definition enables us to obtain generalizations of several classical results. For example, we prove a Hurewicz-type theorem (\ref{Hurewicz_for_A-homology}) relating the $A$-homotopy groups with the $A$-homology groups.

We also give a homological version of the Whitehead theorem which states that, under certain hypotheses, a map between CW-complexes which induces isomorphisms in the $A$-homology groups is a homotopy equivalence (\ref{A-homology_isomorphism}).

Finally, we define a Hurewicz-type map between the $A$-homotopy groups and the $A$-homology groups and embed it in a long exact sequence \ref{A-Whitehead_exact_sequence} generalizing the Whitehead exact sequence \cite{Whi}.

\medskip

We begin with a simple remark which will be used later.

\begin{rem}
Let $p:(E,e_0)\to (B,b_0)$ be a quasifibration, let $F=p^{-1}(b_0)$ and let $A$ be a CW-complex. Since $p$ induces isomorphisms $p_\ast:\pi_i(E,F,e_0)\to \pi_i(B,b_0)$ for all $i\in\mathbb{N}$ and $\pi_i(E,F,e_0)\cong \pi_{i-1}(P(E,e_0,F),c_{e_0})$ and $\pi_i(B,b_0)\cong \pi_{i-1}(\Omega B,c_{b_0})$ it follows that the induced map $\hat{p}:(P(E,e_0,F),c_{e_0})\to (\Omega B,c_{b_0})$ is a weak equivalence. Thus, $\hat{p}$ induces isomorphisms $\hat{p}_\ast:\pi_i^A(P(E,e_0,F),c_{e_0})\to \pi_i^A(\Omega B,c_{b_0})$.

Since $\pi_i^A(E,F,e_0)\cong \pi_{i-1}^A(P(E,e_0,F),c_{e_0})$ and $\pi_i^A(B,b_0)\cong \pi_{i-1}^A(\Omega B,c_{b_0})$ we obtain that $p_\ast:\pi_i^A(E,F,e_0)\to \pi_i^A(B,b_0)$ is an isomorphism for all $i\in\mathbb{N}$.
\end{rem}

Our definition of $A$-homology groups is inspired by the Dold-Thom theorem.

\begin{definition}
Let $A$ be a CW-complex and let $X$ be a topological space. For $n\in\mathbb{N}_0$ we define the \emph{$n$-th $A$-homology group of $X$} as $$H_n^A(X)=\pi_n^A(SP(X))$$
where $SP(X)$ denotes the infinite symmetric product of $X$.
\end{definition}

\begin{theo}
The functor $H_\ast^A(\_)$ defines a reduced homology theory on the category of (path-connected) CW-complexes.
\end{theo}

\begin{proof}
It is clear that $H_\ast^A(\_)$ is a homotopy functor. If $(X,B,x_0)$ is a pointed CW-pair, then by the Dold-Thom theorem, 
the quotient map $q:X\to X/B$ induces a quasifibration $\hat{q}:SP(X)\to SP(X/B)$ whose fiber is homotopy equivalent to $SP(B)$. Since $A$ is a CW-complex there is a long exact sequence
\begin{displaymath}
\xymatrix@C=20pt{\ldots \ar[r] & \pi_n^A(SP(B)) \ar[r] & \pi_n^A(SP(X)) \ar[r] & \pi_n^A(SP(X/B)) \ar[r] & \pi_{n-1}^A(SP(B)) \ar[r] & \ldots}
\end{displaymath}

It remains to show that there exists a natural isomorphism $H_n^A(X)\cong H_{n+1}^A(\Sigma X)$ and that $H_n^A(X)$ are abelian groups for $n=0,1$. The natural isomorphism follows from the long exact sequence of above applied to the CW-pair $(\cn X,X)$. Note that $H_n^A(\cn X)=0$ since $\cn X$ is contractible and $H_n^A$ is a homotopy functor. The second part follows immediately, since $H_0^A(X)\cong H_{1}^A(\Sigma X)\cong H_{2}^A(\Sigma^2 X)$. The group structure on $H_0^A(X)$ is induced from the one on $H_1^A(X)$ by the corresponding natural isomorphism.
\end{proof}

The proof above encourages us to define the relative $A$-homology groups of a CW-pair $(X,B)$ by $H^A_n(X,B)=\pi^A_n(SP(X/B))$ for $n\geq 1$. As shown before, there exist long exact sequences of $A$-homology groups associated to a CW-pair $(X,B)$.

Federer's spectral sequence can be applied as a first method of computation of $A$-homology groups. Indeed, given a finite CW-complex $A$ and a CW-complex $X$, the associated Federer spectral sequence $\{E^a_{p,q}\}$ converges to the $A$-homotopy groups of $SP(X)$ (note that $\pi_1(SP(X))$ is abelian). In this case we obtain that $E^2_{p,q}=H^{-p}(A,\pi_q(SP(X)))=H^{-p}(A,H_q(X))$ if $p+q\geq 1$ and $p\leq -1$. Moreover, we will show later a explicit formula to compute $A$-homology groups of CW-complexes.

We exhibit now some examples.

\begin{ex} \label{ex_1}
If $A$ is a finite-dimensional CW-complex and $X$ is a Moore space of type $(G,n)$ then $SP(X)$ is an Eilenberg-Mac Lane space of the same type. Hence, by the Federer spectral sequence
$$H_r^A(X)=\pi_r^A(SP(X))=H^{n-r}(A,\pi_n(SP(X)))=H^{n-r}(A,G) \qquad \textrm{for $r\geq 1$}.$$
In particular, $H_r^A(S^n)=H^{n-r}(A,\Z)$.

We also deduce that if $X$ is a Moore space of type $(G,n)$ and $A$ is $(n-1)$-connected, then $H_r^A(X)=0$ for all $r\geq 1$.
\end{ex}

\begin{ex} \label{ex_2}
Let $A$ be a Moore space of type $(G,m)$ (with $G$ finitely generated) and let $X$ be a path-connected CW-complex. As in example \ref{ex_A_homotopy}, for $n\geq 1$, there are short exact sequences of abelian groups
\begin{displaymath}
\xymatrix{0 \ar[r] & \ext(G,H_{n+m+1}(X)) \ar[r] & H_n^A(X) \ar[r] & \hom(G,H_{n+m}(X)) \ar[r] & 0}
\end{displaymath}

As a consequence, if $G$ is a finite group of exponent $r$, then $\alpha^{2r}=0$ for every $\alpha\in H_n^A(X)$.

\end{ex}

It is well known that if a CW-complex does not have cells of a certain dimension $j$, then its $j$-th homology group vanishes. As one should expect, a similar result holds for the $A$-homology groups. Concretely, if $A$ is an $l$-connected CW-complex of dimension $k$ and $X$ is a CW-complex, applying the Federer spectral sequence to the space $SP(X)$ one can obtain that:
\begin{enumerate}
\item If $\dim(X)=m$, then $H_r^A(X)=0$ for $r\geq m-l$.
\item If $X$ does not have cells of dimension less than $m'$, then $H_r^A(X)=0$ for $r\leq m'+l-k$.
\end{enumerate}

Following the idea of example \ref{ex_1} we will show now a explicit formula to compute $A$-homology groups.

\begin{prop}
Let $A$ be a finite-dimensional CW-complex and let $X$ be a connected CW-complex. Then for every $n\in\N_0$, $\displaystyle H^A_n(X)=\bigoplus_{j \in \mathbb{N}} H^{j-n} (A,H_j(X))$.
\end{prop}

\begin{proof}
Since $SP(X)$ has the weak homotopy type of $\displaystyle \prod_{n \in \mathbb{N}} K(H_n(X),n)$ and $A$ is a CW-complex we obtain that
\begin{displaymath}
\begin{array}{lcl}
H^A_n(X) & = & \displaystyle \pi^A_n(SP(X)) \cong \pi^A_n(\prod_{j \in \mathbb{N}} K(H_j(X),j)) \cong \prod_{j \in \mathbb{N}} \pi^A_n( K(H_j(X),j)) \cong \\ & \cong & \displaystyle  \prod_{j \in \mathbb{N}} H^{j-n} (A,H_j(X)) = \bigoplus_{j \in \mathbb{N}} H^{j-n} (A,H_j(X))
\end{array}
\end{displaymath}
where the last isomorphism follows from the Federer spectral sequence.
\end{proof}

Now we show that, in case $A$ is compact, $H_\ast^A$ satisfies the wedge axiom. This can be proved in two different ways: using the definition of $A$-homotopy groups or using the above formula. We choose the first one.

\begin{prop}
Let $A$ be a finite CW-complex, and let $\{X_i\}_{i\in I}$ be a collection of CW-complexes. Then $$H_n^A\left(\bigvee_{i\in I} X_i \right)=\bigoplus_{i\in I}H_n^A(X_i).$$
\end{prop}

\begin{proof}
The space $SP(\underset{i \in I}{\bigvee} X_i)$ is homeomorphic to $\underset{i \in I}{\prod}^w SP(X_i)$ with the weak product topology, i.e. $\underset{i \in I}{\prod}^w SP(X_i)$ is the colimit of the products of finitely many factors. Since $A$ is compact, $\pi_n^A(\underset{i \in I}{\prod}^w SP(X_i))\cong \underset{i \in I}{\bigoplus} \pi_n^A(SP(X_i))$ and the result follows.
\end{proof}

We will prove now some of the main results of this article. We begin with a Hurewicz-type theorem relating the $A$-homology groups with the $A$-homotopy groups.

\begin{theo} \label{relative_Hurewicz_for_A-homology}
Let $A$ be the suspension of a path-connected CW-complex of dimension $k-1\geq 1$ and let $(X,B)$ be an $n$-connected CW-pair with $n\geq k$. Suppose, in addition, that $B$ is simply-connected and non-empty. Then $H_r^A(X,B)=0$ for $r\leq n-k$ and $\pi_{n-k+1}^A(X,B)\cong H_{n-k+1}^A(X,B)$.
\end{theo}

\begin{proof}
By the Hurewicz theorem, $H_r(X,B)=0$ for $r\leq n$ and $H_{n+1}(X,B)\cong\pi_{n+1}(X,B)$. Since $(X,B)$ is a CW-pair, by the Dold-Thom theorem we obtain that $\pi_r(SP(X/B))\cong H_r(X/B) \cong H_r(X,B) = 0$ for $r\leq n$.

Since $A$ is a CW-complex of dimension $k \leq n$, then, $H^A_r(X,B)=\pi^A_r(SP(X/B))=0$ for $r\leq n-k$. Also,
\begin{displaymath}
\begin{array}{lcl}
\pi_{n-k+1}^A(X,B) & = &  [(\cn\Sigma^{n-k}A,\Sigma^{n-k}A);(X,B)] \cong H^{n}(\Sigma^{n-k}A,\pi_{n+1}(X,B)) \cong \\ & \cong & H^{n}(\Sigma^{n-k}A,H_{n+1}(X,B)) \cong H^{n+1}(\Sigma^{n-k+1}A,\pi_{n+1}(SP(X/B))) \cong \\ & \cong & [\Sigma^{n-k+1}A,SP(X/B)] = \pi_{n-k+1}^A(SP(X/B)) = H_{n-k+1}^A(X,B)
\end{array}
\end{displaymath}
where the first and fourth isomorphisms hold by \ref{relative_Hopf-Whitney} and \ref{Hopf-Whitney} respectively.
\end{proof}

Moreover, by naturality of \ref{Hopf-Whitney} and \ref{relative_Hopf-Whitney} it follows that the isomorphism above is the morphism induced in $\pi_n^A$ by the map which is the composition of the quotient map $(X,B)\to(X/B,\ast)$ with the inclusion map $(X/B,\ast) \to (SP(X/B),\ast)$.

\medskip

Clearly, from this relative $A$-Hurewicz theorem we can deduce the following absolute version.

\begin{theo} \label{Hurewicz_for_A-homology}
Let $A$ be the suspension of a path-connected CW-complex with $\dim A = k\geq 2$ and let $X$ be an $n$-connected CW-complex with $n\geq k$. Then $H_r^A(X)=0$ for $r\leq n-k$ and $\pi_{n-k+1}^A(X)\cong H_{n-k+1}^A(X)$.

Moreover, the morphism $i_\ast:\pi_{n-k+1}^A(X)\to \pi_{n-k+1}^A(SP(X))=H_{n-k+1}^A(X)$ induced by the inclusion map $i:X\to SP(X)$ is an isomorphism.
\end{theo}

Thus, the morphism $i_\ast:\pi_n^A(X)\to \pi_n^A(SP(X))$ can be thought as a Hurewicz-type map and will be called \emph{$A$-Hurewicz homomorphism}. Not only is it natural, but also it can be embedded in a long exact sequence, as we will show later (\ref{A-Whitehead_exact_sequence}).

\medskip

We give now a homological version of the Whitehead theorem, which states that, under certain hypotheses, a continuous map between CW-complexes inducing isomorphisms in $A$-homotopy groups is a homotopy equivalence.

\begin{theo} \label{A-homology_isomorphism}
Let $A'$ be a path-connected and locally compact CW-complex of dimension $k-1\geq 0$ such that $H_{k-1}(A')\neq 0$ and let $A=\Sigma A'$. Let $f:X\to Y$ be a continuous map between simply-connected CW-complexes which induces isomorphisms $f_\ast:H^A_r(X)\to H^A_r(Y)$ for all $r\in \N$ and $f_\ast:\pi_i(X)\to \pi_i(Y)$ for all $i\leq k+1$. Then $f$ is a homotopy equivalence.
\end{theo}

\begin{proof}
Replacing $Y$ by the mapping cylinder of $f$, we may suppose that $f$ is an inclusion map and hence $(Y,X)$ is $(k+1)$-connected. We will prove by induction that $(Y,X)$ is $n$-connected for all $n\in\mathbb{N}$.

Suppose that $(Y,X)$ is $n$-connected for some $n\geq k+1$. Then 
$$\pi_{n-k}^A(Y,X)=[(\cn \Sigma^{n-k-1}A,\Sigma^{n-k-1}A),(Y,X)]=0$$
because $\dim (\cn \Sigma^{n-k-1}A)=n$. Moreover, by \ref{relative_Hurewicz_for_A-homology} we obtain that
$$\pi_{n-k+1}^A(Y,X)\cong H_{n-k+1}^A(Y,X)=0.$$

Now, by \ref{relative_Hopf-Whitney},
\begin{displaymath}
\begin{array}{lcl}
0 & = & \pi_{n-k+1}^A(Y,X)=H^{n}(\Sigma^{n-k} A,\pi_{n+1}(Y,X)) = H^{k}(A,\pi_{n+1}(Y,X))= \\ & = & \hom(H_{k}(A),\pi_{n+1}(Y,X))\oplus \ext(H_{k-1}(A),\pi_{n+1}(Y,X)).
\end{array}
\end{displaymath}

Then $\hom(H_k(A),\pi_{n+1}(Y,X))=0$. By the hypotheses on $A'$ it follows that $H_{k-1}(A')=H_k(A)$ has $\mathbb{Z}$ as a direct summand. Hence, $\pi_{n+1}(Y,X)=0$ and thus $(Y,X)$ is $(n+1)$-connected.

In consequence, $(Y,X)$ is $n$-connected for all $n\in\mathbb{N}$. Then the inclusion map $f:X\to Y$ is a weak equivalence and since $X$ and $Y$ are CW-complexes it follows that $f$ is a homotopy equivalence.
\end{proof}

To finish, we will make use of a modern construction of the exact sequence Whitehead introduced in \cite{Whi} to embed the $A$-Hurewicz homomorphism defined above in a long exact sequence. As a corollary we will obtain another proof of theorem \ref{Hurewicz_for_A-homology} together with an extension of it. A different way to obtain the Whitehead exact sequence is given in \cite{BH}.

Let $X$ be a CW-complex and let $\Gamma X$ be the homotopy fiber of the inclusion $i:X\to SP(X)$. Hence, there is a long exact sequence
\begin{displaymath}
\xymatrix{\ldots \ar[r] & \pi_n^A(\Gamma X) \ar[r] & \pi_n^A(X) \ar[r]^(.42){i_\ast} & \pi_n^A(SP(X)) \ar[r] & \pi_{n-1}^A(\Gamma X) \ar[r] & \ldots }
\end{displaymath}
and by definition $\pi_n^A(SP(X))=H_n^A(X)$ and $i_\ast$ is the $A$-Hurewicz homomorphism.

Thus, we have proved the following.
\begin{prop} \label{A-Whitehead_exact_sequence}
Let $A$ and $X$ be CW-complexes. Then, there is a long exact sequence
\begin{displaymath}
\xymatrix{\ldots \ar[r] & \pi_n^A(\Gamma X) \ar[r] & \pi_n^A(X) \ar[r]^(.5){i_\ast} & H_n^A(X) \ar[r] & \pi_{n-1}^A(\Gamma X) \ar[r] & \ldots }
\end{displaymath}
\end{prop}

Using this long exact sequence we will give another proof of theorem \ref{Hurewicz_for_A-homology}. Recall that in \cite{Whi}, given a CW-complex $Z$, Whitehead defines the group $\Gamma_n(Z)$ as the kernel of the canonical morphism $\pi_n(Z^n)\to \pi_n(Z^n,Z^{n-1})$ which by exactness coincides with the image of the morphism $j_\ast:\pi_n(Z^{n-1})\to \pi_n(Z^n)$, where $j:Z^{n-1}\to Z^n$ is the inclusion map. It is known that $\Gamma_n(Z)\cong\pi_n(\Gamma Z)$.

Now let $A$ be a CW-complex of dimension $k\geq 2$ and let $X$ be an $n$-connected topological space with $n\geq k$. Replacing $X$ by a homotopy equivalent CW-complex $Y$ with $Y^n=\ast$, it follows that $\Gamma_r(X)=0$ for $r\leq n+1$. Hence, $\Gamma X$ is $(n+1)$-connected.

Therefore, $\pi_r^A(\Gamma X)=0$ for $r\leq n-k+1$. Thus, from the exact sequence above we obtain that $i_\ast:\pi_{n-k+1}^A(X)\to H_{n-k+1}^A(X)$ is an isomorphism. Moreover, $i_\ast:\pi_{n-k+2}^A(X)\to H_{n-k+2}^A(X)$ is an epimorphism.

Summing up, we have proved the following.

\begin{theo} \label{Hurewicz_for_A-homology_2}
Let $A$ be a path-connected CW-complex with $\dim A = k\geq 2$ and let $X$ be an $n$-connected CW-complex with $n\geq k$. Let $i:X\to SP(X)$ be the inclusion map. Then $H_r^A(X)=0$ for $r\leq n-k$, $i_\ast:\pi_{n-k+1}^A(X)\to H_{n-k+1}^A(X)$ is an isomorphism and $i_\ast:\pi_{n-k+2}^A(X)\to H_{n-k+2}^A(X)$ is an epimorphism.
\end{theo}

\end{document}